\documentclass[11pt,english]{article}
\usepackage[margin = 2.5cm]{geometry}
\usepackage[utf8]{inputenc}
\usepackage{amsthm}
\usepackage{amsmath}
\usepackage{amssymb}
\usepackage{mathtools}
\usepackage{graphicx}
\usepackage[pdftex,colorlinks,backref=page,citecolor=blue,bookmarks=false]{hyperref}
\usepackage{enumitem}
\usepackage{xcolor}

\usepackage{caption}% http://ctan.org/pkg/caption
\usepackage[square,sort,comma,numbers]{natbib} 
\usepackage{setspace}
\usepackage{cleveref}
\theoremstyle{plain}

\newtheorem*{theorem*}{Theorem}
\newtheorem{theorem}{Theorem}[section]
\crefname{theorem}{Theorem}{Theorems}
\Crefname{theorem}{Theorem}{Theorems}

\newtheorem*{lemma*}{Lemma}
\newtheorem{lemma}[theorem]{Lemma}
\crefname{lemma}{Lemma}{Lemmas}
\Crefname{lemma}{Lemma}{Lemmas}

\newtheorem*{claim*}{Claim}

\crefname{claim}{Claim}{Claims}
\Crefname{claim}{Claim}{Claims}

\crefname{proposition}{Proposition}{Propositions}
\Crefname{proposition}{Proposition}{Propositions}

\newtheorem{corollary}[theorem]{Corollary}
\crefname{corollary}{Corollary}{Corollaries}
\Crefname{corollary}{Corollary}{Corollaries}

\newtheorem{conjecture}[theorem]{Conjecture}
\crefname{conjecture}{Conjecture}{Conjectures}
\Crefname{conjecture}{Conjecture}{Conjectures}

\crefname{question}{Question}{Questions}
\Crefname{question}{Question}{Questions}

\newtheorem{observation}[theorem]{Observation}
\crefname{observation}{Observation}{Observations}
\Crefname{observation}{Observation}{Observations}

\newtheorem{example}[theorem]{Example}
\crefname{example}{Example}{Examples}
\Crefname{example}{Example}{Examples}

\crefname{algorithm}{Algorithm}{Algorithm}
\Crefname{algorithm}{Algorithm}{Algorithm}

\theoremstyle{definition}

\crefname{problem}{Problem}{Problems}
\Crefname{problem}{Problem}{Problems}

\crefname{definition}{Definition}{Definitions}
\Crefname{definition}{Definition}{Definitions}

\theoremstyle{remark}

\crefname{remark}{Remark}{Remarks}
\Crefname{remark}{Remark}{Remarks}

\usepackage{xpatch}
\xpatchcmd{\proof}{\itshape}{\normalfont\proofnamefont}{}{}
\newcommand{\proofnamefont}{}
\renewcommand{\proofnamefont}{\bfseries}

\newcommand*\dif{\mathop{}\!\mathrm{d}}

\usepackage{framed}

\newcommand{\remove}[1]{}

\def \cA {\mathcal{A}}

\DeclareMathOperator{\dom}{dom}

\title{Correspondence coloring of random graphs}
\author{Zdeněk Dvořák\thanks{Computer Science Institute (CSI) of Charles University,
Malostransk{\'e} n{\'a}m{\v e}st{\'\i} 25, 118 00 Prague, Czech Republic.
E-mail: \protect\href{mailto:rakdver@iuuk.mff.cuni.cz}{\protect\nolinkurl{rakdver@iuuk.mff.cuni.cz}}.
Supported by the European Research Council (ERC) under the European Union’s Horizon 2020
research and innovation programme (grant agreement No 810115).}\and Liana Yepremyan \thanks{Department of Mathmatics, Emory University, 
Atlanta, Georgia 30322. Email: {\tt lyeprem@emory.edu}.} }
\date{\today}

\begin{document}

\maketitle

\begin{abstract}
    We show that Erd\H{o}s-Rényi random graphs $G(n,p)$ with
    constant density $p<1$ have correspondence chromatic
    number $O(n/\sqrt{\log n})$; this matches a prediction from linear
    Hadwiger's conjecture for correspondence coloring.  The proof follows from a simple sufficient condition for correspondence colorability in terms of the numbers of independent sets.
\end{abstract} 

\section{Introduction}

Hadwiger's conjecture famously states that every $K_k$-minor-free
graph is $(k-1)$-colorable.  This is known to be true when $k\le 6$, as shown by
Robertson, Seymour and Thomas~\cite{robertsonseymourthomas}.  The validity of the conjecture for general $k$
is a widely open and extremely challenging problem.
Let $f(k)$ denote the maximum chromatic number of $K_k$-minor-free graphs, so Hadwidger's conjecture
is the claim that $f(k)=k-1$.  For more than 30 years, the best partial result towards the conjecture
was that $f(k)=O(k\log^{1/2} k)$, which follows from the tight bounds on the average degree of $K_k$-minor-free graphs~\cite{kostomindeg,minor2}.
This barrier was broken in 2019 by Norin, Postle and Song~\cite{norin2019breaking} who improved the bound to $O(k\log^{1/4+o(1)} k)$.
This started a series of improvements culminating in the current best bound of $O(k\log\log k)$ by Delcourt and Postle~\cite{delcourt2021reducing}.
A key element of these improvements is an understanding of the chromatic number of small $K_k$-minor-free graphs.
Indeed, Delcourt and Postle~\cite{delcourt2021reducing} shows that $f(k)=O(f'(k))$, where $f'(k)$ denotes the maximum chromatic 
number of $K_k$-minor-free graphs with $O(k\log^4 k)$ vertices.

The analogous problem has been also studied in the list coloring setting.  A \emph{list assignment} for a graph $G$ is a function $L$
that to each vertex $v$ assigns a list $L(v)$ of available colors, and an \emph{$L$-coloring} of $G$ is a proper coloring $\varphi$
such that $\varphi(v)\in L(v)$ for every $v\in V(G)$.  The minimum integer $s$ such that $G$ can be $L$-colored for every
assigment $L$ of lists of size at least $s$ is called the \emph{list chromatic number} (or \emph{choosability}) of $G$.
Analogously to the proper coloring case, we define $f_l(k)$ as the maximum
of the list chromatic numbers of $K_k$-minor-free graphs.  Hadwiger's conjecture is known to be false in the list
coloring setting; Steiner~\cite{steiner2021improved} proved that $f_l(k)\ge (2-o(1))k$.  On the other hand,
Linear Hadwiger's Conjecture for list coloring proposes that $f_l(k)=O(k)$.  The best partial result is by Postle~\cite{postle2020further},
who proved that $f_l(k)=O(k(\log \log k)^6)$.  Again, an important starting point for the argument is a bound on the list
chromatic number of small graphs.

Dvořák and Postle~\cite{dvovrak2018correspondence} introduced \emph{correspondence coloring} (also known as \emph{DP-coloring}),
a further generalization of list coloring by refining the constraints on coloring adjacent vertices;
this notion makes it possible to use certain inductive arguments not available in the list coloring setting.
The \emph{correspondence assignment} for a graph $G$ is a pair $(L,M)$, where $L$ is a list assignment
and $M=\{M_e:e\in E(G)\}$ assigns a (not necessarily perfect) matching $M_e$ between $\{u\}\times L(u)$ and $\{v\}\times L(v)$
to every edge $e=uv\in E(G)$.  We say that $(L,M)$ is an \emph{$\ell$-correspondence assignment} if $L$ assigns at least $\ell$ colors to each vertex.
An \emph{$(L,M)$-coloring} of $G$ is any assignment $\varphi$ of colors to
vertices of $G$ such that $\varphi(v)\in L(v)$ for every $v\in V(G)$ and for
every $e=uv\in E(G)$, the vertices $(u,\varphi(u))$ and $(v,\varphi(v))$ are
non-adjacent in $M_e$.  The \emph{correspondence chromatic number} of $G$ is the smallest integer $\ell$ such that $G$ has an $(L,M)$-coloring for every
$\ell$-correspondence assignment $(L,M)$.

The correspondence chromatic number has attracted a lot of attention.  Among its interesting properties, let us mention the following
results by Bernshteyn~\cite{corr1,berndp}: Every graph of average degree $d$ has correspondence chromatic number at least $\Omega(d/\log d)$,
while every triangle-free graph of maximum degree $\Delta$ has correspondence chromatic number at most $O(\Delta/\log \Delta)$.
Consequently, every triangle-free $d$-regular graph has correspondence chromatic number $\Theta(d/\log d)$.

Let $f_c(k)$ denote the maximum correspondence chromatic number of $K_k$-minor-free graphs.  Clearly $f_c(k)\ge f_l(k)$, and
thus Hadwiger's conjecture does not hold in the correspondence chromatic number setting.
However, unlike the list coloring case, the best known bound $f_c(k)=O(k\sqrt{\log k})$ still matches the maximum average degree.
One of the reasons why the arguments from the list coloring setting do not apply is that we do not know how to bound the correspondence
chromatic number of small graphs.  Indeed, as Norin (personal communication) pointed out, nontrivial bounds are not known even for dense
random graphs; this is very relevant for the considerations of Hadwiger's conjecture, since the tight lower bounds on the average
degree of $K_k$-minor-free graphs come exactly from such random graphs.

As the main result of this note, we give an upper bound on the correspondence chromatic number of Erd\H{o}s-Rényi random graphs
with fixed density of edges; the bound implies that Linear Hadwiger's conjecture for correspondence coloring holds asymptotially almost surely.
\begin{theorem}\label{thm-ren}
For every fixed positive $p\le 1$, Erd\H{o}s-Rényi random graph $G(n,1-p)$ a.a.s. has correspondence chromatic number at most $O(n/\sqrt{\log n})$. 
\end{theorem}
The \emph{Hadwiger number} $h(G)$ of a graph $G$ is the maximum $k$ such that $K_k$ is a minor of $G$.
For fixed $0<p<1$, we have $h(G(n,p))=\Theta(n/\sqrt{\log n}))$ a.a.s., see~\cite{BCE}.  Hence, Theorem~\ref{thm-ren}
supports Linear Hadwiger's conjecture for correspondence coloring.
\begin{corollary}
For every fixed $0\le p\le 1$, Erd\H{o}s-Rényi random graph $G(n,p)$ a.a.s. has correspondence chromatic number $O(h(G))$.
\end{corollary}
The best lower bound for Theorem~\ref{thm-ren} that we are aware of comes from the well-known fact that for positive $p\le 1$,
the random graph $G(n,1-p)$ a.a.s. has independence number $\Omega(\log n)$, and thus $\chi(G(n,1-p))=\Omega(n/\log n)$.
Let us remark that the list chromatic number of $G(n,1-p)$ is known to be $O(n/\log n)$ a.a.s., see e.g.~\cite{alon1993restricted}.

We have already mentioned the general lower bound of Bernshteyn~\cite{corr1}, stating that
every graph of average degree $d$ has correspondence chromatic number at least $\Omega(d/\log d)$; this also gives $\Omega(n/\log n)$
as a lower bound for the correspondence chromatic number of $G(n,1-p)$ for any positive $p<1$.  The lower bound of Bernshteyn comes from considering
a random $\ell$-correspondence assignment, an assignment $(L,M)$ such that $L(v)=\{1,\ldots,\ell\}$ for each vertex $v$
and $M_e$ is chosen independently uniformly at random among the perfect matchings between $\{u\}\times\{1,\ldots,\ell\}$ and $\{v\}\times\{1,\ldots,\ell\}$
for each edge $e=uv$.  It is natural to ask whether a better analysis would not give an improved lower bound in the
case the graph we color is also random.  However, as our second result, we show that this cannot be the case,
even when we consider the coloring of complete graphs.
\begin{theorem}\label{thm-lbcom}
There exists a function $\ell(n)=\Theta(n/\log n)$ such that if $(L,M)$ is a random $\ell(n)$-correspondence
assignment for the complete graph $K_n$, then $K_n$ is a.a.s. $(L,M)$-colorable.
\end{theorem}
As both extreme cases of the correspondence assignment (matching the list coloring vs. completely random)
for $G(n,1-p)$ only requre $O(n/\log n)$ colors, we suspect that the
upper bound given in Theorem~\ref{thm-ren} can be improved.
\begin{conjecture}
For every fixed positive $p\le 1$, Erd\H{o}s-Rényi random graph $G(n,1-p)$ a.a.s. has correspondence chromatic number $\Theta(n/\log n)$. 
\end{conjecture}
It is our hope that this conjecture will actually turn out to be false, as this would give another interesting
example of a difference of behavior between the list and correspondence chromatic numbers.

\section{Sufficient condition for colorability}

Our proofs are based on a sufficient condition for colorability from a given correspondence assignment that
can be extracted from the proof of Bernshteyn~\cite{corr1}.  Let us start with some definitions.
For a nonnegative integer $n$ and a nonnegative real number $b$, let
$$\binom{n}{\le b}=\sum_{i=0}^{\lfloor b\rfloor} \binom{n}{i}$$
denote the number of subsets of an $n$-element set of size at most $k$.  Let us remark that if $n\ge 3b$, then
\begin{equation}\label{eq-binom}
\binom{n}{\lfloor b\rfloor}\le \binom{n}{\le b}\le 2\binom{n}{\lfloor b\rfloor}.
\end{equation}
The \emph{cover graph} of a correspondence assignment $(L,M)$ is the graph with vertex set $\bigcup_{v\in V(G)}
(\{v\}\times L(v))$ and edge set $\bigcup_{e\in E(G)} M_e$.

For a graph $G$, let $I(G)$ denote the number of (not necessarily largest or
maximal) independent sets in $G$.  A \emph{neighborhood subgraph} of $G$ is any
subgraph of $G$ whose vertices are contained in a neighborhood of a vertex of
$G$.  For example, every neighborhood subgraph of a triangle-free graph is
edgeless.  Let $G$ be a graph of maximum degree $\Delta$. For a nonnegative real number $b$, we say that a subgraph $F$ of $G$ is \emph{$b$-large} if $\binom{|V(F)|}{\le b}>\Delta^{1/3}$.
We say that $G$ is \emph{$b$-IS-rich} if every $b$-large neighborhood subgraph $F$ of $G$ satisfies
$$I(F)\ge 2\binom{|V(F)|}{\le b}.$$
A correspondence assignment for $G$ is \emph{$b$-IS-rich} if the same condition
holds for $b$-large neighborhood subgraphs of its cover graph. Note that
neighborhood subgraphs of the cover graph are isomorphic to neighborhood subgraphs of $G$,
and thus if $G$ is $b$-IS-rich, then so is every correspondence assignment for $G$.

\begin{example}\label{ex-trfree}
If $G$ is a triangle-free graph, then every neighborhood subgraph $F$ of $G$ is edgeless, and letting $s=|V(F)|$, we have $I(F)=2^s$.
Let $b=\tfrac{1}{6}\log_2\Delta$.  If $s\le 2b$, then $\binom{s}{\le b}\le 2^s\le 2^{2b}=\Delta^{1/3}$.
Hence, if $F$ is $b$-large, then $s>2b$ and $I(F)\ge 2^s\ge 2\binom{s}{\le (s-1)/2}\ge 2\binom{s}{\le b}$.
Consequently, $G$ is $(\tfrac{1}{6}\log_2\Delta)$-IS-rich.
\end{example}

The argument of Bernshteyn~\cite{berndp} concerning correspondence chromatic number of graphs of bounded clique number
actually applies to IS-rich graphs, giving the following result.
\begin{theorem}\label{thm-col}
For every function $b:\mathbb{N}\to\mathbb{R}_0^+$ such that $b(\Delta)\le \Delta^{o(1)}$, there exists a function $\ell:\mathbb{N}+\to\mathbb{N}$ such that $\ell(\Delta)=\Theta(\Delta/b(\Delta))$ and the following
claim holds.  If $G$ is a graph of maximum degree $\Delta$ and $(L,M)$ is a $b(\Delta)$-IS-rich $\ell(\Delta)$-correspondence
assignment for $G$, then $G$ is $(L,M)$-colorable.
\end{theorem}
Since the proof is almost identical to the one in~\cite{berndp}, we give it in Appendix.
Let us note the following obvious consequence.
\begin{corollary}\label{cor-col}
For every function $b:\mathbb{N}\to\mathbb{R}_0^+$ such that $b(\Delta)\le \Delta^{o(1)}$, there exists a function $\ell:\mathbb{N}\to\mathbb{N}$ with $\ell(\Delta)=\Theta(\Delta/b(\Delta))$
such that every $b(\Delta)$-IS-rich graph of maximum degree $\Delta$ has correspondence chromatic number at most $\ell(\Delta)$.
\end{corollary}

We are not the first to make a similar observation; Theorem~1.13 in~\cite{bonamy2022bounding} is quite related
and \cite{davies2020graph} provides a substantially more powerful framework for proving this kind of colorability results (e.g., unlike their
approach, IS-richness does not apply to graphs where each neighborhood induces only $o(\Delta^2)$ edges).  However,
we find the formulation above quite useful, as it is often easy to establish IS-richness. In particular, if we establish a lower bound
$I(F)\ge g(|V(F)|)$ valid for every (sufficiently large) neighborhood subgraph $F$ of $G$, the following observation can be used to determine a 
function $b$ such that $G$ is $b(\Delta(G))$-IS-rich,
by selecting $b(\Delta)$ to be maximum such that
the conditions (\ref{cond-main}), (\ref{cond-tub}),
and (\ref{cond-der}) hold.
\begin{observation}\label{obs-ver}
Suppose $g:\mathbb{R}_0^+\to \mathbb{R}_0^+$ is an increasing unbounded derivable
function with $g(0)=0$, and for any integer $\Delta\ge 0$,
let $s_\Delta=\lceil g^{-1}(\Delta^{1/3})\rceil$. 
Let $\Delta_0$ be a non-negative integer and suppose that
a function $b:\mathbb{N}\to\mathbb{N}$ for every integer $\Delta\ge \Delta_0$ satisfies
\begin{align}
    \binom{s_\Delta}{b(\Delta)}&\le \tfrac{1}{4}\Delta^{1/3},\label{cond-main}\\
    b(\Delta)&\le s_\Delta/3\text{, and}\label{cond-tub}\\
    (\log g(s))'&\ge \frac{2b(\Delta)}{s}\text{ for every $s\ge s_\Delta$.}\label{cond-der}
\end{align}
If $G$ is a graph of maximum degree $\Delta\ge \Delta_0$ and every neighborhood subgraph $F$ of $G$ with
$s\ge s_\Delta$ vertices satisfies $I(F)\ge g(s)$, then $G$ is $b(\Delta)$-IS-rich.
\end{observation}
\begin{proof}
Consider any $b$-large neighborhood subgraph $F$ of $G$ with $s$ vertices.
Note that $\binom{s_\Delta}{\le b(\Delta)}\le 2\binom{s_\Delta}{b(\Delta)}\le \Delta^{1/3}$ by (\ref{cond-tub}) and (\ref{eq-binom}) and by (\ref{cond-main}); since $F$ is $b$-large,
this implies $s>s_\Delta$.  Then (\ref{cond-der}) implies
\begin{align*}
    \log \frac{g(s)}{g(s_\Delta)}&=\int_{s_\Delta}^s (\log g(x))' \dif x\ge\int_{s_\Delta}^s \frac{2b(\Delta)}{x} \dif x=2b(\Delta)\log\frac{s}{s_\Delta},
\end{align*}
and thus
$$g(s)\ge g(s_\Delta)\cdot\Bigl(\frac{s}{s_\Delta}\Bigr)^{2b(\Delta)}.$$
For any non-negative integer $i\le s_\Delta/2$, 
we have $\frac{s-i}{s_\Delta-i}\le\bigl(\frac{s}{s_\Delta}\bigr)^2$; indeed, this inequality is equivalent
to $(s+s_\Delta)i\le ss_\Delta$, which holds when $s\ge s_\Delta\ge 2i$.  Since $b(\Delta)-1<s_\Delta/2$ by (\ref{cond-tub}), we have
$$\frac{\binom{s}{b(\Delta)}}{\binom{s_\Delta}{b(\Delta)}}=\prod_{i=0}^{b(\Delta)-1} \frac{s-i}{s_\Delta-i}\le \Bigl(\frac{s}{s_\Delta}\Bigr)^{2b(\Delta)}.$$
Finally, note that $$g(s_\Delta)\ge \Delta^{1/3}\ge 4\binom{s_\Delta}{b(\Delta)}$$
by the definition of $s_\Delta$ and (\ref{cond-main}). Combining these inequalities and using (\ref{eq-binom}) and (\ref{cond-tub}), we obtain
$$I(F)\ge g(s)\ge g(s_\Delta)\cdot\Bigl(\frac{s}{s_\Delta}\Bigr)^{2b(\Delta)}\ge 4\binom{s_\Delta}{b(\Delta)}\cdot \frac{\binom{s}{b(\Delta)}}{\binom{s_\Delta}{b(\Delta)}}
=4\binom{s}{b(\Delta)}\ge 2\binom{s}{\le b(\Delta)}.$$
Therefore, $G$ is $b(\Delta)$-IS-rich.
\end{proof}

Given a lower bound $g(s)$ on the number of independent sets in graphs from the
considered graph class, one typically uses the condition (\ref{cond-main}) to determine $b(\Delta)$,
and then mechanically verifies that conditions (\ref{cond-tub}) and (\ref{cond-der}) hold.

\begin{example}\label{ex-col}
Suppose $G$ is an $(r+1)$-colorable graph for a fixed integer $r\ge 2$, and let $F$ be a neighborhood subgraph with $s$ vertices.
Then $F$ is $r$-colorable, and thus $\alpha(F)\ge s/r$ and $I(F)\ge g(s)$ for $g(s)=2^{s/r}$.
We have $s_\Delta=\lceil g^{-1}(\Delta^{1/3})\rceil=\Theta(r\log \Delta)$.  Choosing $\gamma=\Theta(r\log r)$ so that $(e\gamma)^{1/\gamma}\le\sqrt[2r]{2}$ and $\gamma\ge 2r/\log 2$,
let $b(\Delta)=\lfloor s_\Delta/\gamma\rfloor=\Theta(\log_r \Delta)$.  Since $\gamma>3$, (\ref{cond-tub}) holds.  Moreover,
$$\binom{s_\Delta}{b(\Delta)}\le \left(\frac{es_\Delta}{b(\Delta)}\right)^{b(\Delta)}\le (e\gamma)^{s_\Delta/\gamma}=2^{\frac{s_\Delta}{2r}}=\sqrt{g(s_\Delta)}\le 2\sqrt{\Delta^{1/3}}\le \frac{1}{4}{\Delta^{1/3}}$$
for $\Delta$ sufficiently large; hence (\ref{cond-main}) holds.  Finally, $(\log g(s))'=\frac{\log 2}{r}\ge \frac{2}{\gamma}\ge\frac{2b(\Delta)}{s_\Delta}\ge \frac{2b(\Delta)}{s}$ when $s\ge s_\Delta$,
implying that (\ref{cond-der}) holds.
Hence, Observation~\ref{obs-ver} shows that $G$ is $\Theta(\log_r\Delta)$-IS-rich if its maximum degree $\Delta$ is sufficiently large.
\end{example}
\begin{example}\label{ex-clique}
Suppose $G$ is a graph of clique number at most $r+1$ for a fixed integer $r\ge 2$, and let $F$ be a neighborhood subgraph with $s$ vertices.
Since $\omega(F)\le r$, we have $\alpha(F)\ge s^{1/r}-1$ and $I(F)\ge g(s)$ for $g(s)=2^{s^{1/r}-1}$.  Therefore, $s_\Delta=\Theta(\log^r \Delta)$.
Choosing sufficiently small $b(\Delta)=\Theta\Bigl(\frac{\log \Delta}{r\log\log\Delta}\Bigr)$, for sufficiently large $\Delta$ we have
$$\binom{s_\Delta}{b(\Delta)}\le s_\Delta^{b(\Delta)}\le \frac{1}{4}\Delta^{1/3}.$$
Moreover, if $\Delta$ is sufficiently large, then
\begin{align*}
    (\log g(s))'&=\frac{\log 2}{rs^{1-1/r}}=\Theta\Bigl(\frac{1}{r\log^{r-1}\Delta}\Bigr)\\
    &\ge \Theta\Bigl(\frac{1}{r\log^{r-1}\Delta\log\log\Delta}\Bigr)=\frac{2b(\Delta)}{s_\Delta}\ge \frac{2b(\Delta)}{s}
\end{align*}
when $s\ge s_\Delta$. Hence, Observation~\ref{obs-ver} shows that $G$ is $\Theta\Bigl(\frac{\log\Delta}{r\log\log\Delta}\Bigr)$-IS-rich if its maximum degree $\Delta$ is sufficiently large.
\end{example}

Hence, Corollary~\ref{cor-col} together with Examples~\ref{ex-trfree}, \ref{ex-col} and \ref{ex-clique} gives the following (previously known) results: The correspondence chromatic number of a graph $G$ of maximum degree $\Delta$ is
\begin{itemize}
    \item $O(\Delta/\log \Delta)$ if $G$ is triangle-free,
    \item $O(\Delta/\log_r \Delta)$ if $G$ is $(r+1)$-colorable, and
    \item $O(r\Delta\log\log \Delta/\log \Delta)$ if $\omega(G)\le r+1$.
\end{itemize}

Observe that the neighborhood subgraphs of random correspondence assignment for a dense graph are quite sparse;
in particular for the complete graph $K_n$, they a.a.s. have clique number at most five, and thus Theorem~\ref{thm-col} implies that
$K_n$ is a.a.s. colorable from a random $\Theta(n\log\log n/\log n)$-corespondence assignment.  Theorem~\ref{thm-lbcom} follows
by a slightly more involved analysis.

\section{Independent sets in random graphs}

To prove Theorems~\ref{thm-ren} and~\ref{thm-lbcom}, we need corresponding (a.a.s. valid) lower bounds on the number of independent
sets in random graphs.  Let $$I_p(s,t)=\binom{s}{t}p^{\binom{t}{2}}$$ be the expected number of independent sets of size $t$ in $G(s,1-p)$.
This expression is maximized at around $t=\frac{\log s}{\log 1/p}$, at which point we have
$$I_p(s,t)\ge \left(\frac{s}{t}\right)^tp^{t^2/2}=\left(\frac{\sqrt{s}\log 1/p}{\log s}\right)^t=e^{\frac{1}{2\log 1/p}\log^2 s-o(\log^2 s)}.$$
Hence, we want to show that the number of independent sets is a.a.s. close to the expectation.  For that, we use Suen's concentration
inequality for weakly correlated events.  Let us give definitions necessary to state it.  Let $\cA=\{A_i:i\in I\}$ be a finite set of 0-1 random variables
defined on a common probability space.
A graph $\Gamma$ with vertex set $I$ is a \emph{dependency graph} of $\cA$ if for every disjoint $X,Y\subset I$
such that $\Gamma$ has no edges between $X$ and $Y$, the sets of events $\{A_i:i\in X\}$ and $\{A_i:i\in Y\}$ are independent of each other.
Let
\begin{align*}
\mu(\cA)&=\mathbb{E}\Bigl[\sum_{i\in I} A_i\Bigr]\\
d_i(\Gamma)&=\mathbb{E}\Bigl[\sum_{j:ij\in E(\Gamma)} A_j\Bigr]\text{for each $i\in I$}\\
d(\Gamma)&=\max_{i\in I} d_i(\Gamma)\\
D(\Gamma)&=\mathbb{E}\Bigl[\sum_{ij\in E(\Gamma)} A_iA_j\Bigr].
\end{align*}
We write $\mu$, $d_i$, $d$ and $D$ when $\cA$ and $\Gamma$ are clear from the context.

\begin{theorem}[{Janson~\cite[Theorem 10 with $a=1/2$]{janson1998new}}]\label{thm-suen}
Let $\cA=\{A_i:i\in I\}$ be a finite set of 0-1 random variables
defined on a common probability space and let $\Gamma$ be a dependency graph of $\cA$.
Then
$$\mathbb{P}\left[\sum_{i\in I} A_i\le \frac{1}{2}\mu\right]\le \exp\left(-\min\left(\frac{\mu^2}{32D+8\mu},\frac{\mu}{12d}\right)\right).$$
\end{theorem}
We now apply this inequality to the number of intependent sets.
\begin{lemma}\label{lemma-conc}
For positive integers $s$ and $t$ and a positive real number $p\le 1$ such that
$s\ge 2t^2+t$ and $I_p(s,t)\ge s^2/t^4$, the probability that $G(s,1-p)$
has at most $\tfrac{1}{2}I_p(s,t)$ independent sets of size $t$ is at most
$$\exp\left(-\frac{s^2p}{80t^4}\right).$$
\end{lemma}
\begin{proof}
Consider a randomly chosen graph $G=G(s,1-p)$.  Let $I$ be the set of unordered $t$-tuples of vertices of $G$,
and for $S\in I$, let $A_S=1$ if $S$ forms an independent set in $G$ and $A_S=0$ otherwise.
We have $\mathbb{E}[A_S]=p^{\binom{t}{2}}$.
Let $\cA=\{A_S:S\in I\}$ and let $\Gamma$ be the graph with vertex set $I$ where two $t$-tuples in $I$
are adjacent iff their intersection has size at least two.  Then $\Gamma$ is a dependency graph of $\cA$,
since if $X,Y\subseteq I$ are disjoint and non-adjacent in $\Gamma$, then the values of
variables in $\{A_S:S\in X\}$ and in $\{A_S:S\in Y\}$ are determined by the presence or absence of edges in $G$
on disjoint sets of pairs of vertices.

Clearly $\sum_{S\in I} A_S$ is the number of independent sets of size $t$ in $G$, and $\mu=\mu(\cA)=I_p(s,t)$.  
Note that for non-negative $i\le t-2$, we have
$$\frac{\binom{t}{i}\binom{s}{t-i}}{\binom{t}{i+1}\binom{s}{t-i-1}}=\frac{(i+1)(s-t+i+1)}{(t-i)^2}\ge \frac{s-t}{t^2}\ge 2$$
by the assumptions.  Hence, for any $S\in I$,
\begin{align*}
d_S(\Gamma)&=\sum_{T:ST\in E(\Gamma)} \mathbb{E}[A_T]=p^{\binom{t}{2}}\cdot \sum_{i=2}^{t-1}\binom{t}{i}\binom{s-t}{t-i}\le p^{\binom{t}{2}}\cdot \sum_{i=2}^{t-1}\binom{t}{i}\binom{s}{t-i}\\
&\le p^{\binom{t}{2}}\cdot \binom{t}{2}\binom{s}{t-2}\sum_{i=0}^{\infty} 2^{-i}=2p^{\binom{t}{2}}\binom{t}{2}\binom{s}{t-2}.
\end{align*}
Therefore,
$$\frac{\mu}{d(\Gamma)}\ge \frac{\binom{s}{t}p^{\binom{t}{2}}}{2p^{\binom{t}{2}}\binom{t}{2}\binom{s}{t-2}}=\frac{\binom{s}{t}}{2\binom{t}{2}\binom{s}{t-2}}\ge \frac{(s-t)^2}{t^4}\ge \frac{s^2}{4t^4}.$$
Moreover,
\begin{align*}
D(\Gamma)&=\frac{1}{2}\sum_{S\in I}\sum_{i=2}^{t-1}\sum_{T\in I:|S\cap T|=i} \mathbb{E}[A_SA_T]=\frac{1}{2}\binom{s}{t}p^{2\binom{t}{2}}\sum_{i=2}^{t-1}\binom{t}{i}\binom{s-t}{t-i}p^{-\binom{i}{2}}\\
&\le \frac{1}{2}\binom{s}{t}p^{2\binom{t}{2}-1}\binom{t}{2}\binom{s}{t-2}\sum_{i=0}^{\infty}2^{-i}=\binom{s}{t}p^{2\binom{t}{2}-1}\binom{t}{2}\binom{s}{t-2}.
\end{align*}
Let $D=D(\Gamma)$. If $D\ge\mu$, then
$$\frac{\mu^2}{32D+8\mu}\ge \frac{\mu^2}{40D}\ge \frac{\binom{s}{t}^2p^{2\binom{t}{2}}}{40\binom{s}{t}p^{2\binom{t}{2}-1}\binom{t}{2}\binom{s}{t-2}}=\frac{\binom{s}{t}p}{40\binom{t}{2}\binom{s}{t-2}}\ge \frac{s^2p}{80t^4}.$$
On the other hand, if $D<\mu$, then
$$\frac{\mu^2}{32D+8\mu}\ge \frac{\mu^2}{40\mu}=\frac{I_p(s,t)}{40}\ge \frac{s^2}{40t^4}$$
by the assumptions.  The result now follows by Theorem~\ref{thm-suen}.
\end{proof}

In particular, we have the following consequence regarding the number of independent sets (of any size).  Let $g_p(s)=e^{\frac{1}{3\log 1/p}\log^2 s}$.
\begin{corollary}\label{cor-concall}
Let $p<1$ be a positive real number.  A.a.s., every subgraph of $G(n,1-p)$ with $s\ge \log n(\log\log n)^5$ vertices has more than $g_p(s)$ independent sets.
\end{corollary}
\begin{proof}
Let $G=G(n,1-p)$ be a random graph with the number $n$ of vertices sufficiently large.
Consider first a fixed positive integer $s\ge \log n(\log\log n)^5$ and let $t=\bigl\lceil\frac{\log s}{\log 1/p}\bigr\rceil$.  As we argued at the beginning of the section,
$$I_p(s,t)\ge e^{\frac{1}{2\log 1/p}\log^2 s-o(\log^2 s)}\ge 2g_p(s)$$
when $n$ is large enough.  Also, for $n$ large enough, we have $I_p(s,t)\ge 2g_p(s)\ge s^2$ and $s\ge 2t^2+t$.  Hence, by Lemma~\ref{lemma-conc}, for any subset $X\subseteq V(G)$ of size $s$, the probability
that $I(G[X])\le g_p(s)$ is at most $\exp\bigl(-\frac{s^2p}{80t^4}\bigr)\le n^{-2s}$ when $n$ is large enough.  By the union bound, the probability that $G$ contains an $s$-vertex
subgraph with at most $g_p(s)$ independent sets is at most $\binom{n}{s}n^{-2s}\le n^{-s}$.

Consequently, the probability that for every $s\ge \log n(\log\log n)^5$, every subgraph of $G$ with $s$ vertices has more than $g_p(s)$ independent sets
is at least $$1-\sum_{s=\lceil \log n(\log\log n)^5\rceil}^n n^{-s}=1-o(1).$$
\end{proof}

For the proof of Theorem~\ref{thm-lbcom}, it suffices to use Chernoff concentration inequality in the following weak form.
\begin{theorem}[Chernoff inequality]\label{thm-chernoff}
Suppose $X$ is a sum of independent random 0-1 variables, and let $\mu=\mathbb{E}[X]$.
Then for every $t\ge 2\mu$,
$$\mathbb{P}[X\ge t]\le e^{-t/8}.$$
\end{theorem}

We also need Turán's bound, rephrased in terms of the independence number.
\begin{theorem}\label{thm-tur}
Every $n$-vertex graph of average degree $d$ has independence number at least $\tfrac{n}{d+1}$.
\end{theorem}

Let $A(s,n)=\min(s/15,\log^2 n)$.

\begin{lemma}\label{lemma-avgspa}
There exists a positive integer $n_0$ such that for all positive integers $n\ge n_0$ and $s\le n$ and every positive $p\le n^{-13/14}$,
$$\mathbb{P}[\alpha(G(s,p))\le A(s,n)]\le n^{-3s}.$$
\end{lemma}
\begin{proof}
Let $G=G(s,p)$.  If $sp\le n^{-3/7}$, we claim
that $\alpha(G)\ge s/15\ge A(s,n)$ with probability at least $1-n^{3s}$.  Indeed,
if $\alpha(G)<s/15$, then Theorem~\ref{thm-tur} implies that $G$ has average degree
more than $14s$, and thus it has more than $7s$ edges.  The probability that this happens
is at most
$$\binom{\binom{s}{2}}{7s}p^{7s}\le \left(\frac{es^2/2}{7s}\right)^{7s}p^{7s}\le (sp)^{7s}\le n^{-3s}.$$

Hence, suppose that $sp\ge n^{-3/7}$, and thus $s\ge n^{-3/7}p^{-1}\ge \sqrt{n}$. We claim that
$\alpha(G)\ge\log^2 n\ge A(s,n)$ with probability at least $1-n^{-3s}$.  Indeed, if $\alpha(G)<\log^2 n$, then
 Theorem~\ref{thm-tur} implies that $G$ has average degree more than $\tfrac{n}{\log^2 n}-1\ge 2\sqrt{n}$ when $n$ is sufficiently large.  Hence, $G$ has more than $s\sqrt{n}$ edges.
The expected number of edges of $G$ is $\binom{s}{2}p\le snp/2\le sn^{1/14}/2$,
and thus by Theorem~\ref{thm-chernoff}, the probability this happens is at most
$e^{-s\sqrt{n}/8}\le n^{-3s}$.
\end{proof}

\section{Proofs of the coloring results}

Our main results now straightforwardly follow.
\begin{proof}[Proof of Theorem~\ref{thm-ren}]
Without loss of generality, we can assume that $p<1$, as otherwise the graph $G=G(n,1-p)$ has no edges.
We aim to apply Observation~\ref{obs-ver} with the function $g_p(s)$.  This gives $s_\Delta=\lceil \exp(\sqrt{\log(1/p)\log\Delta})\rceil$.
Let
$$b(\Delta)=\Bigl\lfloor\frac{1}{6}\sqrt{\frac{\log\Delta}{\log 1/p}}\Bigr\rfloor-1.$$  We have
$$\binom{s_\Delta}{b(\Delta)}\le s_\Delta^{b(\Delta)}\le e^{2\sqrt{\log(1/p)\log\Delta}\cdot b(\Delta)}\le \frac{1}{4}\Delta^{1/3},$$
and thus (\ref{cond-main}) holds.  For $\Delta$ large enough, (\ref{cond-tub}) clearly holds as well.
Finally, $$(\log g_p(s))'=\frac{2}{3\log 1/p}\cdot \frac{\log s}{s}\ge \frac{2b(\Delta)}{s}$$ for $s\ge s_\Delta$, and thus (\ref{cond-der}) holds.

Note that a.a.s., the maximum degree $\Delta$ of $G$ is at least $(1-p)n/2$, and for $n$ large enough, we have $s_\Delta\ge \log n(\log \log n)^5$.
Therefore, by Corollary~\ref{cor-concall}, a.a.s. every subgraph $F$ of $G$ with $s\ge s_\Delta$ vertices satisfies $I(F)\ge g_p(s)$.
Then Observation~\ref{obs-ver} implies that $G$ is $b(\Delta)$-IS-rich, and Corollary~\ref{cor-col} implies that
the correspondence chromatic number of $G$ is $O(\Delta/b(\Delta))=O(n/\sqrt{\log n})$.
\end{proof}

\begin{proof}[Proof of Theorem~\ref{thm-lbcom}]
Let $\gamma=21$, let $b(n)=\tfrac{1}{\gamma}\log n$ and let $\ell(n)=\Theta(n/\log n)$ be the corresponding function from the statement of Theorem~\ref{thm-col}.
Consider a random $\ell(n)$-correspondence assignment $(L,M)$ for $K_n$.
By Theorem~\ref{thm-col}, it suffices to prove that a.a.s., $(L,M)$ is $b(n)$-IS-rich.  We say that a set $X$
of vertices of the cover graph $H$ of $(L,M)$ is \emph{plausible} if for every distinct $(u,i),(v,j)\in X$ we have $u\neq v$.
Note that the vertex set of every neighborhood subgraph of $H$ is plausible.  Hence, it suffices to show that a.a.s.,
every plausible set $X$ of size $s$ such that $\binom{s}{\le b(n)}>n^{1/3}$ satisfies $I(H[S])\ge 2\binom{s}{\le b(n)}$.

Consider a plausible set $X$ of size $s$.  Observe that $H[X]$ is a random graph $G(s,1/\ell(n))$, since every distinct $(u,i),(v,j)\in X$
form an edge independently iff they are matched in the uniformly random matching between $\{u\}\times \{1,\ldots,\ell(n)\}$ and
$\{v\}\times \{1,\ldots,\ell(n)\}$.  Note that there are at most $(n\ell(n))^s\le n^{2s}$ plausible sets of size $s$, and thus by
Lemma~\ref{lemma-avgspa}, the probability that there exists a plausible set $X$ of any positive size $s$ such that $\alpha(H[X])\le A(s,n)$
is at most $\sum_{s=1}^\infty n^{2s}\cdot n^{-3s}\le 2n^{-1}$.  Therefore, a.a.s. every plausible set $X$ of size $s$ satisfies
$\alpha(H[X])>A(s,n)$, and thus $I(H[X])>2^{A(s,n)}$.

Consider a plausible set $X$ of size $s$ satisfying $\binom{s}{\le b(n)}>n^{1/3}$.  This implies $2^s\ge \binom{s}{\le b(n)}>n^{1/3}$,
and thus $s\ge \tfrac{1}{3\log 2} \log n\ge 3b(n)$.  Let $m=es/b(n)$.  We have
$$n^{1/3}<\binom{s}{\le b(n)}\le 2\binom{s}{\lfloor b(n)\rfloor}\le 2m^{b(n)},$$
and thus
$$m>\frac{e^{\gamma/3}}{2^{\gamma/\log n}}\ge \frac{1}{2}e^{\gamma/3}\ge 500$$
for $n$ sufficiently large.  Hence,
$$2^{s/15}=2^{\frac{mb(n)}{15e}}\ge 2^{2+b(n)\log_2 m}=4m^{b(n)}\ge 2\binom{s}{\le b(n)}.$$
Moreover, $m\le n$, and thus
$$2^{\log^2 n}\ge 2^{2+b(n)\log_2 m}\ge 2\binom{s}{\le b(n)}.$$
Therefore,
$$I(H[X])>2^{A(s,n)}=2^{\min(s/15,\log^2 n)}\ge 2\binom{s}{\le b(n)},$$
and we conclude that $(L,M)$ is $b(n)$-IS-rich.
\end{proof}

\bibliographystyle{acm}
\bibliography{main}

\begin{thebibliography}{10}

\bibitem{alon1993restricted}
{\sc Alon, N.}
\newblock Restricted colorings of graphs.
\newblock {\em Surveys in combinatorics 187\/} (1993), 1--33.

\bibitem{alon2016probabilistic}
{\sc Alon, N., and Spencer, J.~H.}
\newblock {\em The probabilistic method}.
\newblock John Wiley \& Sons, 2016.

\bibitem{corr1}
{\sc Bernshteyn, A.}
\newblock The asymptotic behavior of the correspondence chromatic number.
\newblock {\em Discrete Mathematics 339\/} (2016), 2680--2692.

\bibitem{berndp}
{\sc Bernshteyn, A.}
\newblock The johansson-molloy theorem for dp-coloring.
\newblock {\em Random Struct. Algorithms 54\/} (2019), 653--664.

\bibitem{BCE}
{\sc Bollob\'as, B., Catlin, P., and Erd\H{o}s, P.}
\newblock Hadwiger's conjecture is true for almost every graph.
\newblock {\em European J. Combin. 1\/} (1980), 195--199.

\bibitem{bonamy2022bounding}
{\sc Bonamy, M., Kelly, T., Nelson, P., and Postle, L.}
\newblock Bounding {$\chi$} by a fraction of {$\Delta$} for graphs without
  large cliques.
\newblock {\em Journal of Combinatorial Theory, Series B 157\/} (2022),
  263--282.

\bibitem{davies2020graph}
{\sc Davies, E., Kang, R.~J., Pirot, F., and Sereni, J.-S.}
\newblock Graph structure via local occupancy.
\newblock {\em arXiv 2003.14361\/} (2020).

\bibitem{delcourt2021reducing}
{\sc Delcourt, M., and Postle, L.}
\newblock Reducing linear {H}adwiger's conjecture to coloring small graphs.
\newblock {\em arXiv 2108.01633\/} (2021).

\bibitem{dvovrak2018correspondence}
{\sc Dvo{\v{r}}{\'a}k, Z., and Postle, L.}
\newblock Correspondence coloring and its application to list-coloring planar
  graphs without cycles of lengths 4 to 8.
\newblock {\em Journal of Combinatorial Theory, Series B 129\/} (2018), 38--54.

\bibitem{janson1998new}
{\sc Janson, S.}
\newblock New versions of {S}uen's correlation inequality.
\newblock {\em Random Structures and Algorithms 13}, 3-4 (1998), 467--483.

\bibitem{kostomindeg}
{\sc Kostochka, A.}
\newblock {Lower bound of the {H}adwiger number of graphs by their average
  degree}.
\newblock {\em Combinatorica 4\/} (1984), 307--316.

\bibitem{norin2019breaking}
{\sc Norin, S., Postle, L., and Song, Z.-X.}
\newblock Breaking the degeneracy barrier for coloring graphs with no {$K_t$}
  minor.
\newblock {\em arXiv 1910.09378\/} (2019).

\bibitem{postle2020further}
{\sc Postle, L.}
\newblock Further progress towards the list and odd versions of {H}adwiger's
  conjecture.
\newblock {\em arXiv 2010.05999\/} (2020).

\bibitem{robertsonseymourthomas}
{\sc Robertson, N., Seymour, P.~D., and Thomas, R.}
\newblock {Hadwiger's conjecture for $K_6$-free graphs}.
\newblock {\em Combinatorica}, 13 (1993), 279--361.

\bibitem{steiner2021improved}
{\sc Steiner, R.}
\newblock Improved lower bound for the list chromatic number of graphs with no
  {$K_t$} minor.
\newblock {\em Combinatorics, Probability and Computing\/} (2021), 1--6.

\bibitem{minor2}
{\sc Thomason, A.}
\newblock An extremal function for complete subgraphs.
\newblock {\em Math. Proc. Camb. Phil. Soc. 95\/} (1984), 261--265.

\end{thebibliography}

\section*{Appendix -- proof of Theorem~\ref{thm-col}}

Let $\overline{\alpha}(F)$ denote the median size of an independent set in a
graph $F$, i.e., the largest integer $m$ such that $F$ has at least $I(F)/2$
independent sets of size at least $m$.  The following is analogous to Lemma~4.3
in~\cite{berndp}.

\begin{observation}\label{obs-median}
Let $G$ be a graph of maximum degree $\Delta$, and let $(L,M)$ be a $b$-IS-rich correspondence
assignment for $G$.  If $F$ is a neighborhood subgraph of the cover graph of $(L,M)$ and $I(F)\ge 2\Delta^{1/3}$,
then $\overline{\alpha}(F)>b$.
\end{observation}
\begin{proof}
Let $s$ denote the number of vertices of $F$ and let $\beta=\binom{s}{\le b}$.
If $\beta\le\Delta^{1/3}$, then $I(F)\ge 2\Delta^{1/3}\ge 2\beta$.
If $\beta>\Delta^{1/3}$, then by $b$-IS-richness, we also have $I(F)\ge 2\beta$.
Since $\beta$ was chosen to be the number of subsets of $V(F)$ of size at most $b$,
it follows that $F$ has at least $I(F)-\beta\ge I(F)/2$ independent sets of size greater than $b$.
Consequently, we have $\overline{\alpha}(F)>b$ as required.
\end{proof}

For a correspondence assignment $(L,M)$, a \emph{partial $(L,M)$-coloring} of $G$ is a function $\varphi$
assigning colors to a subset $\dom(\varphi)$ of vertices of $G$ and satisfying the conditions of $(L,M)$-coloring,
i.e., $\varphi(v)\in L(v)$ for every $v\in \dom(\varphi)$ and for every $e=uv\in E(G)$ such that
$u,v\in\dom(\varphi)$, the vertices $(u,\varphi(u))$ and $(v,\varphi(v))$ are non-adjacent in $M_e$.
A color $c\in L(u)$ is \emph{available at $u$} with respect to a partial $(L,M)$-coloring
if for every $uv\in E(G)$ such that $v\in\dom(\varphi)$, the vertices $(u,c)$ and $(v,\varphi(v))$ are non-adjacent
at $M_e$; i.e., the color $c$ can be used to color $u$ without changing the rest of $\varphi$.
For fixed $\Delta$, let $A_u$ be the event that $u\not\in\dom(\varphi)$ and less than $\Delta^{7/12}$ colors are available at $u$.
The following estimate (Lemma~4.2(a) in~\cite{berndp}) is the main ingredient of Bernshteyn's proof.

\begin{lemma}\label{lemma-a}
For every function $b(\Delta)\le \Delta^{o(1)}$, the following claim holds for every sufficiently large $\Delta$.
Let $H$ be a graph of maximum degree at most $\Delta$, let $\ell\ge 5\Delta/b(\Delta)$ be an even integer, and let
$(L,M)$ be a correspondence assignment for $H$.  Let $u\in V(H)$ be a vertex adjacent to all other vertices of $H$.
If $(L,M)$ is $b(\Delta)$-IS-rich, $|L(u)|=\ell$ and $\varphi$ is a partial $(L,M)$-coloring of $H$ such that
$u\not\in \dom(\varphi)$ chosen uniformly at random, then $$\text{Pr}[A_u]<\frac{1}{8\Delta^3}.$$
\end{lemma}
\begin{proof}
By renaming the colors, we can without loss of generality assume that $L(u)=\{1,\ldots,\ell\}$ and that for
each $c\in L(u)$ and each $v\in V(H)\setminus\{u\}$, if there exists an edge $(u,c)(v,c')\in M_{uv}$, then $c=c'$.
Consider the following way of choosing a random partial $(L,M)$-coloring of $H$:
\begin{itemize}
\item Let $\varphi_0$ be chosen uniformly at random among the partial $(L,M)$-colorings of $H$ such that $u\not\in \dom(\varphi)$.
\item For $c=1,\ldots, \ell$:
\begin{itemize}
\item Let $R_c=\{v\in V(H)\setminus\{u\}:(u,c)(v,c)\in M_{uv}\}$.
\item Let $\varphi'_{c-1}$ be the partial $(L,M)$-coloring of $H$ obtained from $\varphi_{c-1}$ by uncoloring
all vertices in $R_c$ that have color $c$, i.e., $\dom(\varphi'_{c-1})=\dom(\varphi_{c-1})\setminus \{v\in R_c\cap \dom(\varphi_{c-1}):\varphi_{c-1}(v)=c\}$.
\item Let $P_c$ be the set of vertices $v\in R_c\setminus \dom(\varphi'_{c-1})$ such that $c$ is available at $u$ with respect to
$\varphi'_{c-1}$.
\item Let $F_c$ be the subgraph of the cover graph of $(L,M)$ induced by $\{(v,c):v\in P_c\}$.
\item Let $S_c$ be an independent set in $F_c$ chosen uniformly at random.
\item Let $\varphi_c$ be obtained from $\varphi'_{c-1}$ by coloring by $c$ all vertices $v$ such that $(v,c)\in S_c$.
\end{itemize}
\item Let $\varphi=\varphi_\ell$.
\end{itemize}
Observe that when $\varphi_{c-1}$ is distributed uniformly, then so is $\varphi_c$.  Hence, the resulting coloring $\varphi$
is uniformly chosen among the partial $(L,M)$-colorings of $H$ such that $u\not\in \dom(\varphi)$.

Let us now define a random sequence $a_1$, $a_2$, \ldots as follows.  Let $c_1<c_2<\ldots<c_k$ be the sequence of all colors $c_i$ such that
in the described process, we have $I(F_{c_i})\ge 2\Delta^{1/3}$.  For $1\le i\le k$, let $a_i=1$ if $|S_{c_i}|\ge\overline{\alpha}(F_{c_i})$
and $a_i=0$ otherwise.  Since at least half of the independent sets in $F_{c_i}$ have size at least $\overline{\alpha}(F_{c_i})$
and $S_{c_i}$ is chosen uniformly at random, we have $\text{E}[a_i]\ge 1/2$ for each $i$. For $i>k$, choose $a_i\in \{0,1\}$ uniformly independently at random.
Then $a_1$, $a_2$, \ldots are independent random variables and by Chernoff inequality,
$$\text{Pr}\Bigl[\sum_{i=1}^{\ell/2} a_i<\ell/5\Bigr]\le \exp(-\Omega(\ell))\le \frac{1}{16\Delta^3}$$
for $\Delta$ sufficiently large.

Similarly, for $1\le j\le \ell-k$, let $c'_j$ be the $j$-th smallest color such that $I(F_{c'_j})<2\Delta^{1/3}$,
let $b_j=1$ if $S_{c'_j}=\emptyset$ and $b_j=0$ otherwise.  Since $S_{c'_j}$ is chosen uniformly at random from less than $2\Delta^{1/3}$ independent
sets, we have $\text{E}[b_j]>\frac{1}{2\Delta^{1/3}}$ for each $j$.  For $j>\ell-k$, choose $b_j\in \{0,1\}$ independently at random
so that $\text{E}[b_j]=\frac{1}{2\Delta^{1/3}}$.  Again using Chernoff inequality,
$$\text{Pr}\Bigl[\sum_{i=1}^{\ell/2} b_i<\Delta^{7/12}\Bigr]\le \text{Pr}\Bigl[\sum_{i=1}^{\ell/2} b_i<\frac{\ell}{3\Delta^{1/3}}\Bigr]\le \exp(-\Omega(\ell/\Delta^{1/3}))\le \frac{1}{16\Delta^3}$$
for $\Delta$ sufficiently large.

Consequently, with probability at least $1-\tfrac{1}{8\Delta^3}$, we have
\begin{align}
\sum_{i=1}^{\ell/2} a_i&\ge\ell/5\label{eq-a}\\
\sum_{i=1}^{\ell/2} b_i&\ge \Delta^{7/12}.\label{eq-b}
\end{align}
If $k\ge \ell/2$, (\ref{eq-a}) and Observation~\ref{obs-median} would imply that there are at least $\ell/5$ indices $i\le k$ such that
$|S_{c_i}|\ge\overline{\alpha}(F_{c_i})>b(\Delta)$,
and thus $\dom(\varphi)\ge \sum_{c=1}^\ell |S_c|>\ell b(\Delta)/5\ge \Delta\ge |V(H)\setminus\{u\}|$,
which is a contradiction.  Therefore, $k\le \ell/2$, and thus (\ref{eq-b}) implies that
there are at least $\Delta^{7/12}$ indices $j\le \ell-k$ such that $S_{c'_j}=\emptyset$.
These colors $c'_j$ are available at $u$ with respect to $\varphi$, and thus $A_u$ is false.
Hence, $\text{Pr}[\lnot A_u]\ge 1-\tfrac{1}{8\Delta^3}$ as required.
\end{proof}

\begin{corollary}\label{cor-a}
For every function $b(\Delta)\le \Delta^{o(1)}$, the following claim holds for every sufficiently large $\Delta$.
Let $G$ be a graph of maximum degree $\Delta$, let $\ell\ge 5\Delta/b(\Delta)$ be an even integer, and let
$(L,M)$ be a $b(\Delta)$-IS-rich $\ell$-correspondence assignment for $G$.  Let $u$ be a vertex of $G$ and let $\varphi_0$
be a partial $(L,M)$-coloring of $G-N[u]$.
If a partial $(L,M)$-coloring $\varphi$ of $G$ is chosen uniformly at random, then
$$\text{Pr}[A_u|\text{$\varphi_0$ is the restriction of $\varphi$}]<\frac{1}{8\Delta^3}.$$
\end{corollary}
\begin{proof}
Let $H$ be the subgraph of $G$ induced by $N[u]$, for each $v\in V(H)$, let $L'(v)$ be the set of colors available
at $u$ with respect to $\varphi_0$, and let $M'=\{M'_e:e\in E(H)\}$, where for $e=xy$, $M'_e$ is the subgraph of $M_e$ induced
by $(\{x\}\times L'(x))\cup (\{y\}\times L'(y))$.  Selecting uniformly at random a partial $(L,M)$-coloring $\varphi$ of $G$
such that $\varphi_0$ is the restriction of $\varphi$ and $u\not\in \dom(\varphi)$ and taking the restriction of this coloring to $H$
results in a uniformly random partial $(L',M')$-coloring $\psi$ of $H$ such that $u\not\in \dom(\psi)$.  Moreover, the
validity of $A_u$ depends only on this restriction to $H$.  Hence, the bound follows from Lemma~\ref{lemma-a}.
\end{proof}

We also need another (simpler) probability bound.  For a partial $(L,M)$-coloring $\varphi$ and a set $S$ of vertices, let $B_S$ be the
event that $S\cap \dom(\varphi)=\emptyset$ and $A_x$ is false for every $x\in S$.
\begin{lemma}\label{lemma-b}
The following claim holds for every sufficiently large $\Delta$.
Let $G$ be a graph, let $(L,M)$ be a correspondence assignment for $G$, let $u$ be a vertex of $G$, and let $S$
be a subset of neighbors of $G$ of size $s=\lceil \Delta^{7/12}\rceil$.  
Let $\varphi_0$ be a partial $(L,M)$-coloring of $G-N[u]$.
Suppose $\varphi$ is a uniformly randomly chosen partial $(L,M)$-coloring of $G$.
Then $$\text{Pr}[B_S|\text{$\varphi_0$ is the restriction of $\varphi$}]<\frac{1}{8\Delta^3\binom{\Delta}{s}}.$$
\end{lemma}
\begin{proof}
Observe that if $A_x$ is false for every $x\in S$ and $S\cap \dom(\varphi)=\emptyset$, then at least $s$ colors
are available at each vertex of $S$ with respect to $\varphi$.  Hence, it suffices to show that
for every partial $(L,M)$-coloring $\varphi_1$ of $G-S$ such that $\varphi_0$ is the restriction of $\varphi_1$ and at least $s$ colors
are available at each vertex of $S$ with respect to $\varphi_1$, we have
$$\text{Pr}[S\cap \dom(\varphi)=\emptyset | \text{$\varphi_1$ is the restriction of $\varphi$}]<\frac{1}{8\Delta^3\binom{\Delta}{s}}.$$
Note that
\begin{equation}\label{eq-bnd}
(s!)^2\ge \exp\bigl(2s\log \tfrac{s}{e}\bigr)\ge \exp\bigl(\tfrac{8}{7}s\log \Delta\bigr)\ge 8\Delta^{s+3}
\end{equation}
for $\Delta$ sufficiently large.  Observe that $\varphi_1$
can be extended to a partial $(L,M)$-coloring of $G$ by giving colors to some vertices of $S$ in at
least $s!$ ways, and by (\ref{eq-bnd}),
$$s!\ge \frac{8\Delta^{s+3}}{s!}\ge 8\Delta^3\binom{\Delta}{s}.$$
Since all vertices of $S$ are uncolored in exactly one of them, the desired inequality follows.
\end{proof}

For a partial $(L,M)$-coloring $\varphi$ and a vertex $u$, let $B_u$ be the event that there are at least $\Delta^{7/12}$ neighbors $x$ of
$u$ such that $A_x$ is false and $x\not\in\dom(\varphi)$.  Using the union bound, Lemma~\ref{lemma-b} has the following consequence.
\begin{corollary}\label{cor-b}
The following claim holds for every sufficiently large $\Delta$.
Let $G$ be a graph of maximum degree $\Delta$ and let $(L,M)$ be a correspondence assignment for $G$.
Let $u$ be a vertex of $G$ and let $\varphi_0$ be a partial $(L,M)$-coloring of $G-N[u]$.
If $\varphi$ is a partial $(L,M)$-coloring of $G$ chosen uniformly at random, then
$$\text{Pr}[B_u|\text{$\varphi_0$ is the restriction of $\varphi$}]<\frac{1}{8\Delta^3}.$$
\end{corollary}

We use the lopsided version of Lovász Local Lemma in the following form~\cite{alon2016probabilistic}.
\begin{lemma}\label{lemma-lll}
Let $d$ be a positive integer.  Let $I$ be a finite set and for each $i\in I$, let $C_i$ be a random event.
Suppose that for every $i\in I$, there is a set $N_i\subseteq I$ of size at most $d$ such that
for every $Z \subseteq I \setminus N_i$,
$$\text{Pr}\Bigl[C_i|\bigwedge_{j\in Z} \lnot C_j\Bigr]\le \frac{1}{4d}.$$
Then
$$\text{Pr}\Bigl[\bigwedge_{i\in I} \lnot C_i\Bigr]>0.$$
\end{lemma}

For a vertex $u\in V(G)$, let $N^k[u]$ denote the set of vertices at distance at most $k$ from $u$.

\begin{proof}[Proof of Theorem~\ref{thm-col}]
Let $\Delta_0\ge 2$ be an integer such that Corollaries~\ref{cor-a} and~\ref{cor-b} hold for $\Delta\ge\Delta_0$,
and let us define $\ell(\Delta)=\max(\Delta_0,2\lceil \tfrac{5}{2}\Delta/b(\Delta)\rceil)$.  Let $G$ and $(L,M)$ be as in
the statement of the theorem.  If $\Delta<\Delta_0$, then $G$ can be $(L,M)$-colored greedily, since $\ell(\Delta)\ge \Delta_0>\Delta$.
Hence, suppose that $\Delta\ge \Delta_0$.

Consider any vertex $u\in V(G)$ and sets $Z,Z'\subseteq V(G)\setminus N^3[u]$.  Let $d=2\Delta^3$ and note that
$2|N^3[u]|\le d$.  The validity of $A_x$ for $x\in Z$
and $B_y$ for $y\in Z'$ depends only on the restriction of the partial coloring to $G-N[u]$, and thus Corollaries~\ref{cor-a} and~\ref{cor-b}
imply that for a partial $(L,M)$-coloring $\varphi$ of $G$ chosen uniformly at random,
\begin{align*}
\text{Pr}\Bigl[A_u|\bigwedge_{x\in Z} \lnot A_x\land \bigwedge_{y\in Z'} \lnot B_y\Bigr]&\le \frac{1}{4d}\\
\text{Pr}\Bigl[B_u|\bigwedge_{x\in Z} \lnot A_x\land \bigwedge_{y\in Z'} \lnot B_y\Bigr]&\le \frac{1}{4d}.
\end{align*}
By Lemma~\ref{lemma-lll}, with non-zero probability both $A_u$ and $B_u$ are false for every $u\in V(G)$.
This implies that each uncolored vertex has at least $\Delta^{7/12}$ available colors, but less than
$\Delta^{7/12}$ uncolored neighbors.  Hence, $\varphi$ can be greedily extended to an $(L,M)$-coloring of $G$.
\end{proof}

\end{document}